\documentclass[10pt, a4paper]{amsart}
\usepackage{amssymb, amsmath, amsfonts, amsthm, verbatim}
\usepackage{amsaddr}
\usepackage[numbers]{natbib}
\usepackage{tikz}
\usepackage{fullpage}

\newtheorem{lemma}[equation]{Lemma}

\newtheorem{theorem}[equation]{Theorem}

\newtheorem{conjecture}[equation]{Conjecture}
\newtheorem{example}[equation]{Example}

\newcommand{\LL}{\mathcal{L}}
\newcommand{\wh}{\widehat{\LL(G)}}

\newcommand{\M}{M\"obius }
\newcommand{\sg}{\leqslant}

\newcommand{\sdp}{\colon\!}

\begin{document}

\title{The order complex of $PGL_2(p^{2^n})$ is\\contractible when $p$ is odd}

\author{Emilio Pierro}
\address{\textnormal{Department of Mathematics,\\
London School of Economics and Political Science,\\
Houghton Street,\\
London, WC2A 2AE}}
\email{e.pierro@mail.bbk.ac.uk}
\date{\today}

\begin{abstract}
Given a group $G$, its lattice of subgroups $\mathcal{L}(G)$ can be viewed as a simplicial complex in a natural way.
The inclusion of $1_G, G \in \mathcal{L}(G)$ means that $\mathcal{L}(G)$ is contractible, and so the topology of the order complex $\widehat{\mathcal{L}(G)} := \mathcal{L}(G) \setminus \{1_G,G\}$ is usually determined instead.
In this short note we consider the homotopy type of $\widehat{\mathcal{L}(G)}$ where $G \cong PGL_2(p^{2^n})$, $p \geq 3$, $n \geq 1$ and show that $\widehat{\mathcal{L}(G)}$ is contractible.
This is consistent with a conjecture of Shareshian on the homotopy type of order complexes of finite groups.
\end{abstract}

\maketitle

\section{Introduction}
Given a poset $P$, the order complex $\Delta(P)$ is the simplicial complex whose vertices are the elements of $P$ and more generally whose $k$-cells are chains of length $k+1$ in $P$.
In the case where $P$ is the subgroup lattice $\LL(G)$ of a finite group $G$, this yields a contractible space and so its topology is rather uninteresting.
For a finite group $G$, then, we consider the order complex of $\wh := \LL(G) \setminus \{G,1_G\}$.
The following example serves to illustrate that if $G$ and $H$ are two groups with identical composition factors, then their order complexes can vary wildly.

\begin{example} Let $G$ be a group with $|G| = p^2$ where $p$ is prime.
If $G$ is cyclic, then $\wh$ is contractible.
If $G$ is elementary abelian, then $\wh$ is a wedge of $p$ spheres, that is, homotopy equivalent to the disjoint union of $p+1$ points.
\end{example}

In the case where $G$ is a solvable group, it was shown by Shareshian \cite{sharshell} that the order complex of $G$ is shellable if and only if $G$ is solvable.
Recall that a shellable topological space is homotopy equivalent to a wedge of equidimensional spheres.
This led to the following conjecture of Shareshian \cite[Conjecture 1.1 (A)]{shartop}.

\begin{conjecture}[Shareshian]
Let $G$ be a finite group and let $H < G$.
Then, the order complex of $\widehat{\LL(H,G)} = \{K \mid H<K<G\}$ has the homotopy type of a wedge of spheres.
\end{conjecture}

If $G$ has a non-trivial Frattini subgroup then a result of Quillen, which we state in the following section, shows that the order complex of $G$ is contractible.
Thus, the determination of the homotopy type $\wh$ when $G$ has simple composition factors and trivial Frattini subgroup is wide open.
Hence it seems natural to consider the case when $G$ is an almost simple group.
We remind the reader that a group $G$ is almost simple if there exists a simple group $T$ such that $T \sg G \sg$ Aut$(T)$.

\smallskip

The homotopy types of minimal simple groups (i.e. those which do not contain simple subgroups) were also considered by Shareshian \cite[Lemma 3.8]{sharshell}.
To the best of the author's knowledge, the homotopy type of $\wh$ is known in only a small number of other cases, such as those appearing in the work of Kramarev and Lokutsievski\u\i \cite{kramlok}.
In unpublished work of Shareshian and the author, they determined that when $G \cong R(3) \cong PSL_2(8) \sdp 3$, $\wh$ is homotopy equivalent to a wedge of $504 = |G'|$ $2$-spheres.

\smallskip

In this short note we determine an infinite family of almost simple groups order complex is contractible.
\begin{theorem} \label{prop}
Let $G \cong PGL_2(p^{2^n})$ where $p \geq 3$ is prime and $n\geq1$. Then $\wh$ is contractible.
\end{theorem}

We do not consider the groups $PGL_2(2^n)$ since these are much harder to attack owing to the the abundance of subfield subgroups which complicates the determination of the homotopy type.
In the case $G \cong PSL_2(4)$, the homotopy type of $G$ was determined by Shareshian \cite{sharshell} and it is homotopy equivalent to a wedge of $60$ $1$-spheres.
The full automorphism group of $PSL_2(q^2)$, where $q$ is odd, contains a non-split extension of the socle.
We mention that the non-split extension $PSL_2(9).2$ is isomorphic to a point-stabiliser of $M_{11}$ in its natural permutation representation on $11$ points.
It is immediate \cite[Proposition 1.8]{kt2} that when $G$ is isomorphic to the non-split extension $PSL_2(q^2).2$, $\wh$ is contractible.
On the contrary, the extensions $P\Sigma L_2(q^l)$ seem too difficult to approach in general.

\smallskip

We follow the notation of \cite{atlas} and explicitly mention the following conventions.
The notation $nX$ is used to denote a conjugacy class of elements of order $n$ and where $|C_G(nA)| \geq |C_G(nB)|$, etc.
The notation $2A_1B_2$ denotes an elementary abelian group of exponent $2$ containing one element from the conjugacy class $2A$ and two elements from the conjugacy class $2B$.

\section{Proof of Proposition \ref{prop}} \label{proof}
In general, to consider the full poset $\wh$ for a finite group $G$ would be quite challenging.
The first reduction we can make follows from Quillen's Fiber Lemma \cite{quillen} which we state in the following specific form.

\begin{lemma}[Quillen] Let $f \colon X \to X$ be the inclusion map of the poset $X$ such that $f(x) \leq x$ for all $x \in X$.
Let $f/x = \{x \in X | f(x) \leq x\}$, the ``fiber'' of $x$.
If $f/x$ is contractible for all $x \in X$, then $f$ is a homotopy equivalence.
\end{lemma}

For a finite group $G$ this allows us to restrict $\wh$ to the subposet consisting of subgroups which occur as the intersection of maximal subgroups.
By abuse of notation we shall denote this subposet by $\wh$.
Dually, we could also consider the subposet consisting of all subgroups generated by their ``minimal'', i.e. prime order, elements; although in practice this subposet seems substantially harder to determine.
We note that in the case of $PGL_2(q)$, where $q \equiv 1$ mod $4$, this would allow us to omit the maximal parabolic subgroups, since all of their prime order elements are contained in $PSL_2(q)$.
Quillen's Lemma would not allows us to simply take the intersection of these two dual posets.
One would have to first take the ``usual'' subposet consisting of intersections of maximal subgroups, and then consider the minimal elements of this subposet, which need not necessarily still contain the prime order subgroups.

\smallskip

The connections between the order complex of $\wh$ and the \M function of $G$ are very close; for example $\mu_G(1_G) + 1 = \chi(\wh)$ \cite[Theorem 3]{ROTA}.
However, these connections are quite natural once the \M function is correctly interpreted in terms of Algebraic Topology.
In particular, for the determination of the \M function of $G$ it is sufficient to know which subgroups $H \sg G$ occur as intersections of maximal subgroups, otherwise $\mu(G,H)=0$ \cite{hall}.
This was determined for the groups $PGL_2(q)$ by Downs in his thesis \cite{phdowns}, which, unfortunately the author has not been able to obtain.

\smallskip

Nevertheless, the maximal subgroups of $G \cong PGL_2(q)$ can be found in \cite[Table 8.7]{bhrd} and since we restrict ourselves to the case where $G$ does not contain maximal subfield subgroups, we are able to easily determine the subgroups necessary.
It is worth mentioning that the case of determining $\wh$ when $G$ is an almost simple group is often more tractable than in the case where $G$ is simple since we are able to use the following result of Bj\"orner and Walker \cite[Theorem 1.1]{bjwa}.

\begin{theorem}[Bj\"orner--Walker] \label{complem} If $L$ is a bounded lattice, $s \in \hat{L}$ and the complements of $s$ form an antichain, then
\[\hat{L} \simeq \bigvee_{x \perp s} \left( \sum((\hat{0},x)*(x,\hat{1})) \right).\]\end{theorem}

We use the above result in conjunction with the following specific case of a result of Kratzer and Th\'evenaz \cite[Proposition 1.6]{kt2}.

\begin{lemma}[Kratzer--Th\'evenaz] \label{link}
Let $G$ be a finite group with order complex $\wh$ and let $H \sg G$.
If the subposet $(H,G) = \{K \mid H<K<G\}$ $($or dually $(1_G,H) = \{K \sg G \mid 1<K<H\})$ is contractible, then $\wh$ is homotopy equivalent to $\wh \setminus \{H\}$.
\end{lemma}

We now consider the case where $G \cong PGL_2(q)$, where $q=p^{2^n}$, $p \geq 3$ is an odd prime and $n \geq 1$.
These groups contain a unique conjugacy class of involutions in $G \setminus \textnormal{soc}(G)$ \cite[Table 4.5.1]{gls3} which we denote $2B$, and the centraliser in $G$ of an element $x \in 2B$ is isomorphic to $D_{2(q+1)}$.
The subgroups $\langle x \rangle$ where $x \in 2B$ will be our set of complements in the application of Lemma \ref{complem}.
We shall then show that the subcomplex $(\langle x \rangle,G) = \{H \sg G \mid \langle x \rangle < H < G\}$ is contractible so that, by Lemma \ref{link}, removing the subgroups $\langle x \rangle$ from $\wh$ does not change the homotopy type of $\wh$.
This will yield that $\wh$ has the homotopy type of an empty wedge of spheres, in other words, a single point, and so $\wh$ is contractible.

It remains, then, to prove the following.

\begin{lemma} Let $G \cong PGL_2(q)$ where $q=p^{2^n}$, $p \geq 3$ and $n \geq 1$.
Let $2B$ denote the unique conjugacy class of involutions in $G \setminus \textnormal{soc}(G)$ and let $x \in 2B$.
The ascending link of $x$ is contractible.
\end{lemma}

\begin{proof}
Let $x$ be a fixed element of $2B$ and let $C:=C_G(x) \cong D_{2(q+1)}$.
The maximal subgroups of $G$ are known \cite[Table 8.7]{bhrd} and those which contain $x$ are isomorphic to $D_{2(q-1)}$ or $D_{2(q+1)}$.
For any pair of dihedral subgroups of $G$, a counting argument shows that the their rotation subgroups intersect trivially, and so the only possible non-trivial intersection between two such dihedral groups is $\langle x \rangle$ or a four group whose non-trivial elements are $x$ and a pair of involutions from the class $2A$.

Hence, the ascending link of $\langle x \rangle$ consists of subgroups of these three isomorphism types.
An easy calculation shows these are:\begin{enumerate}
\item $1+(q+1)/2$ subgroups isomorphic to $D_{2(q+1)}$, one of which is $C$;
\item $(q+1)/2$ subgroups isomorphic to $D_{2(q-1)}$; and,
\item $(q+1)/2$ subgroups of shape $2A_1B_2$.
\end{enumerate}
Since $C_G(2A) \cong D_{2(q-1)}$, it is clear that in the ascending link of $\langle x \rangle$, each maximal subgroup isomorphic to $D_{2(q-1)}$ contains a unique subgroup of shape $2A_1B_2$, and so these can be removed without changing the homotopy type.
Similarly, each of the subgroups isomorphic to $D_{2(q+1)}$ except for $C_G(x)$ can be removed without changing the homotopy type.
Finally, we are left with the $(q+1)/2$ subgroups of shape $2A_1B_2$, contained in the unique maximal subgroup $C_G(x)$.
We see, then, that the whole ascending link of $\langle x \rangle$ can be retracted into $\langle x \rangle$ itself, and so $\wh$ is homotopic to a wedge of $|G:C|$ copies of suspensions of a point.
Hence $\wh$ is contractible.
\end{proof}

\section*{Acknowledgement}
The author wishes to thank John Shareshian for helpful comments in the preparation of this paper.

\bibliographystyle{amsplain}
\providecommand{\bysame}{\leavevmode\hbox to3em{\hrulefill}\thinspace}
\providecommand{\MR}{\relax\ifhmode\unskip\space\fi MR }
\providecommand{\MRhref}[2]{%
  \href{http://www.ams.org/mathscinet-getitem?mr=#1}{#2}
}
\providecommand{\href}[2]{#2}

\end{document}